 \documentclass[a4paper,12pt]{amsart}

\usepackage{graphicx}

\usepackage{amsfonts}
 \usepackage{amssymb}
 \usepackage{amsmath,amsxtra,amsthm}

\usepackage[utf8]{inputenc}

\usepackage[mathscr]{eucal}

 \usepackage{dsfont}

\usepackage{hyperref}

\usepackage{wrapfig}

\usepackage[all]{xy}

\usepackage{amssymb}
\usepackage{amsbsy}
 \usepackage{amsmath}
\usepackage{subfigure}
\usepackage{color}
\usepackage{enumitem}
\usepackage{easybmat}
\usepackage{xypic}
\usepackage{comment}
\usepackage{amsthm}
\theoremstyle{definition}
\newtheorem{Definition}{Definition}[section]
\newtheorem{Theorem}[Definition]{Theorem} 
\newtheorem{Corollary}[Definition]{Corollary}
\newtheorem{Lemma}[Definition]{Lemma}
\newtheorem{Remark}[Definition]{Remark}
\newtheorem{Proposition}[Definition]{Proposition}
\newtheorem{Example}[Definition]{Example}

\newcommand{\leaf}[0]{{N}}

\newcommand{\R}{\mathbb{R}}

\newcommand{\rd}{\mathrm{d}}

\newcommand{\leg}{\mathfrak{Leg}}

\usepackage{MnSymbol,wasysym}

\begin{document}

\title[]{On Dirac structures admitting \\ a variational approach}

\author[O.Cosserat]{Oscar~Cosserat} 
\address{Oscar~Cosserat, LaSIE  -- CNRS \& University of La Rochelle,
Av. Michel Cr\'epeau, 17042 La Rochelle Cedex 1, France}
\email{oscar.cosserat@univ-lr.fr}

\author[A.Kotov]{Alexei~Kotov}
\address{Alexei~Kotov, Faculty of Science, University of Hradec Kralove, Rokitanskeho 62, Hradec Kralove 50003, Czech Republic}
\email{oleksii.kotov@uhk.cz}

\author[C.Laurent]{Camille~Laurent-Gengoux}
\address{Camille~Laurent-Gengoux, Institut Elie Cartan de Lorraine (IECL), UMR 7502 --  3 rue Augustin Fresnel, 57000 Technop\^ole Metz, France}
\email{camille.laurent-gengoux@univ-lorraine.fr}

\author[L.Ryvkin]{Leonid~Ryvkin}
\address{Leonid~Ryvkin, Georg-August-Universität Göttingen,
Institut für Mathematik, Bunsenstr. 3-5, 37073 Göttingen // Institut Camille Jordan, Université Claude Bernard Lyon 1, 
43 boulevard du 11 novembre 1918, 69622 Villeurbanne France}
\email{leonid.ryvkin@mathematik.uni-goettingen.de}

\author[V.Salnikov]{Vladimir~Salnikov} 
\address{Vladimir~Salnikov, LaSIE  -- CNRS \&  La Rochelle University,
Av. Michel Cr\'epeau, 17042 La Rochelle Cedex 1, France}
\email{vladimir.salnikov@univ-lr.fr}

\begin{abstract}
We discuss the notion of  {horizontal} cohomology for Dirac structures and, more generally, Lie algebroids. We then use this notion  {to describe the condition allowing} a variational formulation of Dirac dynamics.

\end{abstract}

\maketitle 
\tableofcontents

\section*{Introduction / motivation}

This paper is a part of series of works by the authors (in various combinations) concerning rather broad subjects of ``geometrizing mechanics'' and ``geometric integrators''. They include attempts to spell-out the underlying geometric structures for a large class of mechanical systems, like it was done decades ago with symplectic 
structures for Hamiltonian (conservative) systems. They also address the question of application of these geometric construction to design reliable simulation tools for respective classes of mechanical systems. 

The precise question we ask ourselves in this paper is mostly motivated by the results of \cite{RSHD}, where the appearance of (almost) \emph{Dirac structures} for mechanical systems with constraints is discussed. It can be vaguely formulated as: ``given a Dirac structure, what else do we need to know to define meaningful dynamics on it''. {The question is in the spirit of the paper \cite{morse}, where the notion of \emph{Dirac systems} is described in the context of constrains as well as for control theory; it also somehow complements the series of works \cite{dirac-algebroids, GG} on a uniform description, using algebroids, of constraint systems in both Hamiltonian and Lagrangian formalisms. All those works, like many others are inspired by the approach to mechanics using double vector bundles introduced in \cite{tulcz, tulcz2}. In this paper, more precisely, we study the cohomological  {conditions} for a system  arising from a Dirac structure to admit a variational (Lagrangian) formulation.}  {For convenience, by some language abuse, we will call them \emph{obstructions}, but what we actually mean is \emph{sufficient conditions}.  That is for ``good cases''}, when this obstruction is absent, we explain how a Lagrangian is constructed. This includes some classes of Poisson structures, for which a variational formulation of Hamiltonian mechanics becomes possible.  This is also an important step to the construction of Dirac-structure-preserving numerical methods, since having constructed the Lagrangian, under some assumptions, one can profit from the well-established machinery of variational numerical methods. 

The paper is organized as follows: 
We start by recalling some notions of Lie algebroids and give the definition of their cohomology. Then, we describe the main geometric tool -- cohomology of Dirac structures, providing some examples. In the second part we explain the relation of this Dirac cohomology and the obstructions to construct a variational formulation for the dynamics on the Dirac structures.
We illustrate the construction on some examples and non-examples.
To conclude, we explain some ideas about variational integrators and possible application of those in our setting -- this is a separate rich topic that we intend to elaborate in another more ``mechanically oriented'' paper.

\section{The  {horizontal} cohomology of Lie algebroids and Dirac structures}

\subsection{Lie algebroids}
The notion of Lie algebroids is a simultaneous generalization of tangent bundles and Lie algebras. In this subsection we briefly review the relevant notions. We refer to \cite{Mackenzie} for a detailed account.

\begin{Definition}
Let $M$ be a smooth manifold. A Lie algebroid $(A,\rho, [\cdot,\cdot])$ is given by a finite-dimensional vector bundle $A$, a vector bundle morphism $\rho:A\to TM$, called anchor and a ($\mathbb R$-bilinear) Lie bracket on the sections of $A$ $$[\cdot,\cdot]:\Gamma(A)\times \Gamma(A)\to \Gamma(A)$$ satisfying for all $f\in C^\infty(M)$, $s,s'\in \Gamma(A)$:
$$[s,fs']=f[s,s']+\rho(s)(f)\cdot s'.$$
\end{Definition}

It can be shown that the above condition implies that $\rho_*:\Gamma(A)\to \Gamma(TM)=\mathfrak X(M)$ is a Lie algebra homomorphism. Lie algebroids appear in many different settings:
\begin{itemize}
    \item The tangent bundle $TM$ with its usual bracket and $\rho=id$ is a Lie algebroid.
    \item Let $F\subset TM$ be an involutive subbundle, i.e. a foliation. Then $F$ is a Lie algebroid with the restricted bracket and the inclusion $F\to TM$ as anchor.
    \item Let $\mathfrak g$ be a Lie algebra and $v:\mathfrak g\to \mathfrak X(M)$ an infinitesimal action (i.e. a Lie algebra homomorphism). Then $\mathfrak g\times M$ is a Lie algebroid with bracket induced by the Lie bracket on $\mathfrak g$ and anchor $\rho(\xi, p)=v(\xi)(p)$. In particular Lie algebras can be seen as Lie algebroids over a point.
    \item Let $\pi \in \Gamma(\Lambda^2TM)$ be a Poisson bivector. Then the cotangent bundle carries a Lie algebroid structure induced by $\pi$. This is actually a particular instance of the Lie algebroid associated to a Dirac structure, which we will treat in the next subsection.
\end{itemize}

Lie algebroids can be alternatively defined as fiberwise linear Poisson structures on vector bundles or as differential graded manifolds of degree 1 (cf. e.g. \cite{vaintrob}). In particular, there is a degree 1 differential (the Lichnerowicz differential, \cite{lich}) $d_A:\Gamma(\Lambda^\bullet A^*)\to \Gamma(\Lambda^{{\bullet+1}} A^*)$,  {where $\bullet$ denotes an appropriate integer index. This differential is defined by} 
\begin{align*}
   (d_A\eta)(\xi_1,...,\xi_{n+1})=&\sum_{i}(-1)^{i+1}\rho(\xi_i)(\eta(\xi_1,...,\hat{\xi_i},..,\xi_{n+1}))\\+&\sum_{i<j}(-1)^{i+j}\eta([\xi_i,\xi_j],\xi_1,...,\hat{\xi_i},...,\hat{\xi_j},..,\xi_{n+1})
\end{align*}

The differential satisfies $d_A^2=0$ and induces a cohomology, which is called Lie algebroid cohomology and denoted by $H^\bullet(A)$. The anchor $\rho$ induces a morphism  {from the usual de Rham cohomology to it:}  $H_{dR}^\bullet(M)\to H^\bullet(A)$.\\

A Lie algebroid always induces a singular foliation on $M$: The subspace $\rho(A)\subset TM$ is always involutive and -- by construction -- locally finitely generated, hence the integrability theorem (cf. \cite{Hermann},   {reviewed in \cite{sylvain}}) applies and $M$ has a decomposition into immersed connected submanifolds $M=\bigsqcup_\alpha \leaf_\alpha$ such that $T\leaf_\alpha=\rho(A)|_{\leaf_\alpha}$ for all $\leaf_\alpha$. Moreover, the bracket on $A$ restricts to well-defined brackets on $A|_{\leaf_\alpha}$, turning $A|_{\leaf_\alpha}\to \leaf_\alpha$ into Lie algebroids.\\ 

The submanifolds $\leaf_\alpha$ are called leaves  (of the foliation induced by the Lie algebroid) and a Lie algebroid is called transitive, if it has only one leaf, i.e. $\rho(A)=TM$ and $M$ is connected.

\subsection{The  {horizontal} cohomology of Lie algebroid}

\begin{Definition}
Let $A\overset{\rho}{\to} TM$ be a Lie algebroid over the smooth manifold $M$. We define:
\begin{itemize}
    \item The subspace of $\rho$-horizontal forms at $m\in M$ as:
    $$(\Lambda^\bullet A_m^*)^{hor}:=\left\{
    \alpha \in \Lambda^\bullet A_m^*~ |~ \iota_v\alpha=0~\forall v\in \mathrm{ker}(\rho_m:A_m\to T_mM) 
    \right\}$$
    \item The subspaces of $\rho$-{horizontal} forms:
  $$  \Gamma(\Lambda^\bullet A^*)^{ {hor}}=\left\{
    \alpha \in  \Gamma(\Lambda^\bullet A^*)~ |~ \alpha_m\mathrm{~and~} (d_A\alpha)_m~\mathrm{~are~horizontal~for~all~}m
    \right\}
    $$
    \item the  {horizontal} cohomology of $A$ as the quotient
    
    $$\mathcal H^\bullet_{{hor}}(A)=\frac{\mathrm{ker}(d_A:  \Gamma(\Lambda^\bullet A^*)^{{hor}}\to  \Gamma(\Lambda^{{\bullet+1}} A^*)^{{hor}})}{\mathrm{Image}(d_A:  \Gamma(\Lambda^{{\bullet-1}} A^*)^{{hor}}\to  \Gamma(\Lambda^\bullet A^*)^{{hor}})}$$
\end{itemize}
\end{Definition}

\begin{Remark} Of course, there are natural maps $H_{dR}^\bullet(M)\to \mathcal H^\bullet_{{hor}}(A) $ and $\mathcal H^\bullet_{{hor}}(A)\to H^\bullet(A)$. In general, these maps are neither injective nor surjective, as we will see in the sequel.
\end{Remark}

\begin{Example}
When $A$ is a transitive Lie algebroid (i.e. $\rho(A)=TM$), then $H^\bullet_{{hor}}(A)$ is isomorphic to the usual de Rham cohomology $H_{dR}^\bullet(M)$. More generally, if $\rho(A)$ is a regular foliation (i.e. if $\rho$ has constant rank), then $H^\bullet_{{hor}}(A)$ recovers the longitudinal cohomology of the foliation induced by $\rho(A)$.
\end{Example}
The above example actually extends to the following:

\begin{Lemma}\label{lem:oidpot}
Let $A$ be a Lie algebroid and $\leaf\subset M$ a leaf of $A$ and $\eta\in \Gamma((\Lambda^kA^*)^{hor})$ a $\rho$-horizontal form. 
\begin{enumerate}
    \item $\eta|_\leaf$ is a $\rho$-horizontal $k$-form on the restricted Lie algebroid $A|_\leaf\to \leaf$, i.e. it induces a unique k-form $\eta_\leaf\in \Omega^k(\leaf)$. 
    \item $\eta$ is completely determined by the collection $\{\eta_\leaf ~|~\leaf\mathrm{~leaf~of~}A\}$.
    \item When $\eta$ is  {horizontal}, we have $(d_A\eta)_\leaf=d\eta_\leaf$.
    \item Let $[\eta]=0\in \mathcal H^k_{{hor}}(A)$, then $[\eta_\leaf]=0\in H^k_{dR}(\leaf)$ for all leaves $\leaf$ of the algebroid $A$.
\end{enumerate}
\end{Lemma}

\begin{Remark}
For $A=F$ the Lie algebroid of a regular foliation,
$\mathcal H^\bullet_{{hor}}(A)$ is 
the longitudinal cohomology along the leaves of $A$, which must not be confused with the cohomology of the leaf space, or equivariant cohomology, that is the cohomology of forms which are basic with respect to the leaf space.
\end{Remark}

\begin{Remark}
In view of the above example and remark, let us stress 
that the usual intuition related to the ``{horizontal}'' as  {parallel to} the base should be applied very carefully since it is sometimes misleading.  {This newly defined cohomology also should not be confused with basic or equivariant cohomology of algebroids (cf. \cite{ginzburg, zucchini}).}
\end{Remark}

\subsection{Dirac structures}
In this subsection, we briefly review the notion of Dirac structures. For details we refer to \cite{zbMATH00004959}. \\

Let $M$ be a manifold. The standard Courant
algebroid (exact Courant algebroid with vanishing \v{S}evera class) on $M$ is given by $(\mathbb TM, \langle\cdot,\cdot\rangle, [\cdot,\cdot])$, where $\mathbb TM$ is, as a vector bundle $TM\oplus T^*M$, $\langle\cdot,\cdot\rangle:\mathbb TM\otimes \mathbb TM\to \mathbb R$ is the standard symmetric pairing $\langle(v,\alpha),(w,\beta)\rangle=\alpha(w)+\beta(v)$ and $[\cdot,\cdot]:\Gamma(\mathbb TM)\otimes \Gamma(\mathbb TM)\to \Gamma(\mathbb TM)$ is the Courant bracket:
$$
[(X,\alpha),(Y,\beta)]=([X,Y],L_X\beta-L_Y\alpha -\frac{1}{2} d(\beta(X)-\alpha(Y))).
$$
This bracket is skew-symmetric but does not satisfy the Jacobi identity. There is an alternative definition of bracket (the Dorfman bracket), which satisfies the Jacobi identity, but is not skew symmetric. We are now prepared to give the central definition of a Dirac structure:

\begin{Definition}
A $dim(M)$-dimensional 
subbundle $D\subset \mathbb TM$ is called Dirac structure, if it is isotropic (i.e. $\langle D,D\rangle=0$) and involutive (i.e. $[\Gamma(D),\Gamma(D)]\subset \Gamma(D)$ ). 
\end{Definition}

Let us look at some examples:
\begin{itemize}
    \item Let $\omega\in\Omega_{cl}^2(M)$ be a closed 2-form. Then its graph $\Gamma_\omega=\{(v,\iota_v\omega) ~|~v\in TM\}$ is a Dirac structure. Any Dirac structure with bijective anchor 
    $D\to TM$ (i.e. the restriction of the projection $\mathbb TM\to TM$ to $D$ is bijective) can be described by the graph of a closed 2-form.
    \item Let $\pi\in\Gamma(\Lambda^2TM)$ be a Poisson structure. Its graph $\Gamma_\pi=\{(\iota_\alpha\pi,\alpha) ~|~\alpha\in T^*M\}$ is a Dirac structure. Any Dirac structure with bijective projection $D\to T^*M$ can be described as a Poisson bivector.
    \item Let $F\subset TM$ be an involutive (regular) distribution and $F^\circ\subset T^*M$ its annihilator. Then $D=F\oplus F^\circ$ is a Dirac structure.
\end{itemize}

\begin{Remark}
The closedness of the 2-form $\omega$ is essential for $\Gamma_\omega$ to be involutive. However, for a non-closed 2-form, we can consider a twisted Courant algebroid $(\mathbb TM,\langle\cdot,\cdot\rangle, [\cdot,\cdot]_{d\omega})$ (with a twisting of the Courant bracket using $d\omega$) with respect to which $\Gamma_\omega$ is involutive, i.e. a (twisted) Dirac structure. In this article, we will only work with Dirac structures in the standard Courant algebroid.
\end{Remark}

Restricted to a Dirac structure $D$, the Courant bracket becomes a Lie bracket and turns $D$ into a Lie algebroid. We call the ({horizontal}) Lie algebroid cohomology of $D$ its ({horizontal}) Dirac cohomology. The Dirac structure also induces a canonical  {horizontal} 2-cocycle that we will now describe.

\subsection{The natural  {horizontal} two-cocycle of a Dirac structure}
Let $D\subset \mathbb  {TM}$ be a 
Dirac structure. We define $\omega_D\in\Gamma(\Lambda^2D^*)$ by $$\omega_D((v,\alpha),(w,\beta))=\alpha(w)-\beta(v).$$  
 As $D$ is isotropic, we have $\omega_D((v,\alpha),(w,\beta))=2\alpha(w)=-2\beta(v)$, i.e. $\omega_D$ is horizontal  {at each point}.
A computation based on the involutivity of $D$ (\cite{burs}), shows that $\omega_D$ is closed in Dirac cohomology, i.e. $d_D\omega_D=0$, and hence $\omega_D$ is  {horizontal}.\footnote{  
 {Note that from now on we write  {horizontal} as a shorthand for $\rho$-horizontal, since the anchor map is no longer explicitly used.}}
It thus yields a natural class in $\mathcal H^2_{{hor}}(D)$.
Hence, in view of Lemma \ref{lem:oidpot} we have:

\begin{Lemma}Let $D\subset \mathbb TM$ be a Dirac structure.
\begin{enumerate}
    \item There is a naturally induced  {horizontal} cocycle $\omega_D\in \Gamma(\Lambda^2D^*)^{{hor}}$ 
associated to any Dirac structure $D$.
\item If $[\omega_D]=0\in \mathcal{H}_{{hor}}^2(D)$, then for any leaf $\leaf$ of $D$, $[(\omega_D)_\leaf]=0\in H^2_{dR}(\leaf)$.
\end{enumerate}
\end{Lemma}

\begin{Remark}
The second statement above heavily relies on the fact that we work with the  {horizontal} cohomology: Even when a primitive of $\omega_D$ in $\Gamma(\Lambda^\bullet D^*)$ exists, it has no reason to be horizontal (i.e. to restrict to leaves) in general.
\end{Remark}

\begin{Remark} \label{almost-dirac}
The construction of the above 2-form works out, even when $D$ is an \emph{almost-Dirac structure}, i.e. an isotropic $dim(M)$-dimensional subbundle of $\mathbb TM$ (which might not be involutive). However, in this case $D$ can fail to be a Lie algebroid and hence there is no associated Lie algebroid cohomology to lie in.
\end{Remark}

\subsection{Examples} \label{sec:examples}
Let us interpret this class in the most important cases.

\begin{Example} \label{sympl-ex}
$D=\Gamma_{\omega}\subset \mathbb TM$ is the graph of a (pre-)symplectic structure $\omega$. Then the Lie algebroid structure on $D$ is isomorphic to the Lie algebroid $TM$. Hence, the  {horizontal} cohomology is canonically isomorphic to the de Rham cohomology ($\mathcal{H}_{{hor}}^\bullet(D)\cong H^\bullet(D)=H^\bullet_{dR}(M)$). The form $\omega_D$ corresponds to $\omega$ under this isomorphism.
\end{Example}

\begin{Example}
\label{reg-fol}
Let $D=F\times F^\circ\subset \mathbb TM$, where $F\subset TM$ is a regular foliation. The Lie algebroid structure on $D$ is the induced bracket on $F$ (the usual Lie bracket of vector fields) and the zero bracket on $F^\circ$.
Then $H^\bullet(D)\cong H^\bullet(F)\times \Lambda^\bullet (F^\circ)^*$ and $\mathcal H^\bullet_{{hor}}(D)\cong H^\bullet(F)$. The form $\omega_D$ is zero.
\end{Example}

\begin{Remark} \label{almost-holonomic}
If $F\subset TM$ is not involutive, then $D=F\times F^\circ$ is still an almost-Dirac structure (cf. Remark \ref{almost-dirac}). Even though there is no cohomology, the associated 2-form $\omega_D$ is still zero. 
\end{Remark}

\begin{Example}
Let $D\subset \mathbb TM$ be the graph of a Poisson structure $\pi$. Then the Lie algebroid $D$ is isomorphic to $T^*M$ and $H^\bullet(D)\cong H^\bullet_{\pi}(M)$ is known as the Poisson cohomology {(see for example \cite{dSw})}. The class of $\omega_D$ in $H^\bullet(D)$ corresponds to the class of $\pi$ in $H^2_{\pi}(M)$. The class of $\omega$ in the finer cohomology $\mathcal H_{bas}^\bullet(D)$ is zero if and only if $\pi\in \mathfrak X^2(M)$ admits a primitive $E\in \mathfrak X(M)$ {(a vector field $E$ satisfying $L_E \pi = \pi$)}, which is tangent to the Poisson structure,
i.e. is a section of $\rho(D)\subset TM$. 

For instance for $M=\mathbb R^2$, the Poisson structure $\pi=x^2\partial_x\wedge \partial_y$ admits such a primitive $E=x\partial_x$.

Also, on each leaf $\leaf$ of $D$, $\omega_D$ restricts to a symplectic form (\cite{burs}). { And hence there are no exact symplectic structures on compact manifolds, this leads to the following obstruction to the existence of horizontal primitives of $\omega_D$: 
for a Poisson structure with  vanishing cohomology class, the only compact leaves of its symplectic foliation are points. }

\end{Example}

\begin{Example}
{Here is a classical type of Poisson structures:}
let $\mathfrak g$ be a Lie algebra. Its dual $M=\mathfrak g^*$ carries a natural Poisson structure, whose leaves are the coadjoint orbits of $G\curvearrowright\mathfrak g^*$ (cf. e.g. \cite{zbMATH06054532}). 

In this case $[\pi]\in H^2_\pi(M)=H^2(D)$ is always zero: There exists a (linear) vector field $E$, such that $[\pi,E]=\pi$. However, $E$ can rarely be chosen to be tangent to the coadjoint orbits. For instance, when $\mathfrak g$ is compact and semi-simple, this can never occur.

\end{Example}

\section{A variational approach to exact Dirac structures}
\subsection{Dirac paths}

\begin{Theorem} \label{thm1}
Let $D \subset \mathbb TM$ be a Dirac structure over $M$, 
$H\in C^\infty(M)$ be a Hamiltonian function and $\gamma$ a path on $M.$ 

Assume that the {horizontal} 2-class $[\omega_D]$ vanishes, and let $\theta\in \Gamma(D^*)^{hor}$ be such that $d_D\theta=\omega_D$, then the following statements are equivalent:
\begin{itemize}
    \item[(i)] The path $\gamma$ is a Hamiltonian curve, i.e. $(\dot \gamma(t), dH_{\gamma(t)})\in D$ for all $t$. 
    \item[(ii)] 
    All Dirac paths $\zeta:I\to D$ over $\gamma$ (i.e. $\rho(\zeta)=\dot \gamma$) 
    are critical points among the Dirac paths with the same end points  
     of the following functional:
    \begin{equation}
    \zeta\mapsto \int_I\left(\theta_{\gamma(t)}(\zeta(t))+H(\gamma(t))\right)dt
    \label{eq:DiracFunctionnal}
    \end{equation}
     \item[(iii)] There exists a Dirac path $\zeta:I\to D$ over $\gamma$ and the latter is the critical point among Dirac paths with the same end points of the functional \eqref{eq:DiracFunctionnal}.
\end{itemize}
 
\end{Theorem}

\begin{proof}
As $\theta(\zeta)$ does not depend on the choice of a Dirac path $\zeta$ over $\gamma$, the equivalence of  (ii) and  (iii) is obvious. \\

For the equivalence between (i) and (iii), let us first note that we can restrict the functional \eqref{eq:DiracFunctionnal} to a fixed leaf $\leaf$ containing $\gamma$. On $\leaf$, $\omega_{\mathcal D}$ corresponds to a pre-symplectic form $\omega_\leaf$ and $\theta$ yields a potential of  $\omega_\leaf$, i.e. a $1$-form $\theta_\leaf \in \Omega^1(\leaf)$ such that $d \theta_\leaf = \omega_\leaf$.  
Thus, it suffices to prove the assertion of the theorem for an exact pre-symplectic form $\omega_\leaf ={d} \theta_\leaf $. That is to say, we want to prove, that $\omega_\leaf^\flat(\dot \gamma)=dH_\gamma$ if and only if $\gamma: I \to \leaf$ is a critical point of 
$$\int_I \big(\theta_\leaf|_{\gamma(t)} (\dot \gamma(t))+H(\gamma(t)) \big ) dt$$
among curves with fixed endpoints. For symplectic manifolds this is a classical result ~\cite{arnold}, we provide a proof (that follows the same lines) for the sake of completeness.

Take a small variation 

$$v : \begin{array}{lrcl}
I\times (-\epsilon,\epsilon) & \to & \leaf \\
(t,s) & \mapsto & v(t,s) \end{array}$$ of $\gamma$ with fixed endpoints and set $Y(t) = \frac{\partial}{\partial s}\big|_{s=0}v \in T_{\gamma(t)} \leaf $. 
We call $\gamma^* L_Y\theta_\leaf$ the smooth $\mathbb R $-valued function on $I$ defined by
$$ \gamma^* L_Y\theta_\leaf(t) = \frac{\partial}{ \partial t} \theta_\leaf |_{\gamma(t)} ( Y(t) ) + d \theta_\leaf (Y(t) , \dot \gamma(t) )  .  $$
(If $Y$ happens to be the restriction to $\gamma $ of a vector field $Y$ on $\leaf$, $\gamma^* L_Y\theta_\leaf(t)dt \in \Omega^1(I)$ is the pull-back through $\gamma: I \to \leaf $ of $L_Y  \theta_\leaf $, hence the notation.)

\begin{align*}
& \left. \frac{\partial}{\partial s}\right|_0 \int_{t \in I} \left( \theta_\leaf|_{v_s(t)} \left(\frac{\partial{ v_s}}{\partial t }(t)\right)+H(v_s(t)) \right)dt \\ 
& = \int_{t \in I} \left. \frac{\partial}{\partial s}\right|_0 \left( \theta_\leaf|_{v_s(t)} \left(\frac{\partial{ v_s}}{\partial t }(t)\right)+H(v_s(t)) \right)dt \\
&=  \int_{t \in I}  \left( \gamma^* L_Y\theta_\leaf(t)  +dH|_{\gamma(t)}  (Y(t)) \right)dt
\\ &= \int_{t \in I}  \left(  d \theta_\leaf (Y(t) , \dot \gamma(t) )  +dH|_{\gamma(t)}  (Y(t)) \right)dt  \\
 \end{align*}
In the two last equalities, we used integration by parts and the fact that $Y(0)=0$ and $Y(1)=0$ since the endpoints of $v$ are fixed, so that $ \int_I  \frac{\partial}{ \partial t} \theta_\leaf |_{\gamma(t)} ( Y(t) ) dt =0$.
Since $Y(t)$ is arbitrary except at end points, this proves the result.
\end{proof}

\begin{Remark}
In the above proof, we actually do not require the equality $d\theta=\omega$ and the horizontality of $\theta$ globally. We only need both properties along $L$. 
\end{Remark}

The above enables us to treat a much wider class of Dirac structures than in section \ref{sec:examples}, especially singular ones.

\begin{Example}
Consider $\mathbb R^4$ with the Poisson structure \\$\pi=(x^2+y^2)\partial_x\wedge \partial_y+ \partial_z\wedge \partial_w$. Along the singular leaf $L=\{x=y=0\}$ the vector field $w\partial_w$ is horizontal and a primitive for $\pi$, hence the above remark applies and we obtain a variational characterization of paths. 
\end{Example}

\subsection{Generalized implicit Lagrangian systems.}
\label{ILS}

Let $L\colon TQ \to \R$ a (possibly degenerate) Lagrangian.

{
\begin{Remark}
Note that in this section and in the next one the base manifold will be systematically $Q$ instead of $M$ used before -- this is to stress the fact that in the Lagrangian picture it is necessary to lift the construction to double vector bundles, and thus the relevant Dirac structures will be over $Q$ or $T^*Q$ depending on the context. The general facts about Dirac structures will still be formulated with the base manifold denoted by $M$.
\end{Remark}
}

\begin{Definition} We call Tulczyjew's differential the map $u\mapsto {\mathcal D}_uL := \beta (d_uL)$, where $\beta:T^*TQ\to T^*T^*Q$ is the Tulczyjew isomorphism (cf. Appendix \ref{tul-iso}). Its image is a submanifold of $T^*T^*Q$.
\end{Definition}

\begin{Definition}
We also define a map from $TQ$ to $T^* Q $ by $\mathbb{F}L(v)$ for every $v \in T_q Q$ by
 $$ \frac{\partial}{\partial t}{\Big|_{t=0}} L( v + tw) = \langle \mathbb{F}L(v) , w \rangle $$

\begin{enumerate}
    \item[a)] We denote by $\leg =\mathbb FL(TQ) \subset T^*Q$ the 
image of $\mathbb{F}L$.
    \item[b)] We call partial vector fields on $\leg$ sections\footnote{For $E$ a vector bundle over a manifold $X$ and $Y \subset X$ an arbitrary subset (not necessarily a manifold), we denote by $\Gamma(E)|_Y  $ 
    restrictions to $Y$ of smooth sections of $E$ in a neighborhood of $Y$ in $X$.} of $\Gamma(T(T^*Q))|_{\leg} $.
    \item[c)] 
An integral curve of a partial vector field $X$ on $\leg$   is a path $t \mapsto u_t \in TQ$ such that
 $$ \frac{{\mathrm d} }{{\mathrm d}t} \mathbb{F}L(u_t) =  X_{ \mathbb{F}L(u_t)}  .$$
 \item[d)] An implicit Lagrangian system for an almost Dirac structure  $\mathbb{D} \subset \mathbb{T}T^*Q$ is a pair $(X, L)$, with $X$ a partially defined vector field on $\leg $, such that   $(X(\mathbb FL(u)),  \mathcal D_uL)\in {\mathbb D}$ for all $u$ in $TQ$.
\end{enumerate} 
\end{Definition}

\begin{Remark}
We do not assume partial vector fields on $ \leg$ to be tangent to $\leg$ in any sense.  Of course, if they are not tangent, they may have little integral curves.
\end{Remark}

\begin{Remark}
Particular cases include:   
\begin{itemize}
    \item Usual Hamiltonian dynamics. When $\mathbb FL$ is a diffeomorphism and $\mathbb D$ is the graph of the canonical symplectic form on $T^*Q$, implicit Lagrangian systems are pairs $(X_H,L)$, where $X_H$ is the Hamiltonian vector field of {$H$, the Hamiltonian function associated to $L$ via the Legendre transform.}
\item Constraint dynamics, which actually motivated the construction, we give some details below: Example \ref{holonomic}. 
\end{itemize}

\end{Remark}

\subsection{Implicit Lagrangian systems with magnetic terms}
\label{magnetic}

 \begin{Definition}[\cite{burs}] Let $D \subset \mathbb T M$ be a subbundle.
 \begin{enumerate}
     \item For all $\phi\colon M' \to M $, we denote by  $\phi^!D$ the set 
     $$\phi^! D_{m'}  := \left\{ (X,\phi^* \beta)    \hbox{ with }  X \in T_{m'}M', \beta \in T_{\phi(m')}^*  M\hbox{ s.t.} (\phi_*(X), \beta) \in D_{\phi(m')}  \right\} $$
     When $D$ is an (almost-)Dirac structure we call $\phi^!D$ the pullback of $D$.
     \item Let $\omega$ be a $2$-form $\omega\in \Omega^2(M)$, we denote by $e^\omega D$ the set 
     $$e^\omega D=\{(v,\beta +\iota_v\omega)~|~ (v,\beta)\in D  \}$$
     and call it the gauge transform of $D$. 
 \end{enumerate}
\end{Definition}

\begin{Lemma}[cf. e.g. \cite{burs}] Let $D \subset \mathbb TM$ be a Dirac structure and $M'$ be a manifold. 
\begin{enumerate}
    \item For any smooth map $\phi\colon M' \to M $, $\phi^!D$ is a Dirac structure on $M'$.
    \item For any closed 2-form $\omega \in \Omega^2(M)$, $e^\omega D$ is a Dirac structure on $M$.
\end{enumerate}
\end{Lemma}

Given $D\subset \mathbb TQ$ a Dirac structure on $Q$, this lemma allows to consider \emph{(i)} its pull back $\pi^! D $ on $T^* Q $ through the canonical base map $ \pi\colon T^* Q \to Q$, then \emph{(ii)} consider the gauge transformation $e^\Omega \pi^! D $ of this pull-back with respect to the canonical symplectic $2$-form $\Omega$.

 \begin{Definition}
 Let $D\subset \mathbb TQ$ be a Dirac structure on $Q$. We call constrained magnetic Lagrangian system an implicit Lagrangian system for the Dirac structure $\mathbb D=e^\Omega\pi^!D\subset \mathbb TT^*Q$ as above. 
 \end{Definition}

 \begin{Theorem}
 \label{theorem-ILS}
 Let $D\subset \mathbb TQ$ be a Dirac structure and $L:TQ\to\mathbb R$ a Lagrangian. Assume that the 2-form $\omega_D\in \Gamma(\Lambda^2D^*)^{hor}$ admits a {horizontal} primitive $\theta\in \Gamma(D^*)^{hor}$. Then for $q :I\to Q$ the following are equivalent:
 \begin{enumerate}
     \item[a)] There exists a Dirac path $\zeta: I \to D$ such that $ \rho(\zeta)=\dot q$ which is the critical point among Dirac paths with the same end points of
     \begin{align} \label{eq:ILS-functional} \int_I (L (\rho(\zeta(t))) + \theta (\zeta(t)))  dt .\end{align}
      \item[b)] For all $t\in I$, the following condition holds.
      \begin{align}\label{eq:graphcondition}
          \left(\frac{\partial }{\partial t}\mathbb FL(\dot q(t)), \mathcal D_{\dot q(t)}L\right)\in \mathbb D=e^\Omega\pi^!D.
      \end{align} 
 \end{enumerate}
 
 \end{Theorem}

\begin{proof}
We claim that it suffices to check the equivalence on a small open set in $Q$. The second assertion is clearly local in nature and a path is a critical point for the functional \eqref{eq:ILS-functional} if and only if for any $t\in I$ there is a subinterval $t\in I_t\subset I$ on  which it is critical.\\

It suffices therefore to establish the equivalence on an open subset $U$ of $Q$, on which the Dirac structure takes the following normal form (\cite{blohmann}):

\begin{itemize}
    \item $U=S\times N$, where $S\subset \mathbb R^{a}, N\subset \mathbb R^b$ with $q(0)$ lying in $S\times \{0\}$.
    \item $D|_U=e^\eta((TS \oplus \{0\} )\times \Gamma^\Pi)$, where $\eta\in \Omega^2_{cl}(U)$ and $\Pi\in\mathfrak X^2(N)$ is a Poisson bivector field vanishing at $0$ and $\Gamma^\Pi$ is the corresponding Dirac structure on $N$.
\end{itemize}

As $S\times \{0\}$ is a leaf of $D$, $q(t)\in S\times \{0\}$ for all $t$. By our assumptions on $D$ 
$$ \int_I (L (\rho(\zeta(t))) + \theta (\zeta(t)))  dt
= \int_I (L (\dot q) + \theta^S (\dot q) ) dt,
$$
where $\theta^S\in\Omega^1(S)$ is the one-form on $S$ induced by $\theta$ and hence satisfies $d\theta=i_S^*\eta$ where $i_S:S\to U, i_S(s)=(s,0)$.

We will denote by $L^S$ the restriction of $L$ to $TS$ and write $q=(q^S, 0)$. With these conventions, the above functional reads:
$$
\int_I (L^S (\dot q^S) + \theta^S (\dot q^S))  dt.
$$

The classical Euler-Lagrange theorem {(\cite{Pontryagin})} with magnetic term implies that being a critical point of this functional is equivalent to 

\begin{align}\label{eq:graphons}
\left(\frac{\partial }{\partial t}\mathbb FL^S(\dot q^S(t)), \mathcal D_{\dot q^S(t)}L^S\right)\in e^\Omega\pi^!e^{\eta^S}(TS\oplus \{0\}).
\end{align}
Via the isomorphism $\mathbb TT^*U=\mathbb TT^*S\times \mathbb TT^*N$, the assertion \eqref{eq:graphcondition} decomposes as two conditions, the first one (on $S$) being \eqref{eq:graphons}. The second condition (on $N$) is always satisfied, as one can verify by a straightforward computation in local coordinates which relies on the fact that $\pi_N(q(t))=0$ for all $t$.

\end{proof}

\begin{Corollary}\label{cor:imp}
Let $Q,L,D$ be as in Theorem \ref{theorem-ILS} and $(X,L)$ an implicit Lagrangian system. Then any integral curve $\gamma$ of $X$ is the base path of a critical point of \eqref{eq:ILS-functional}.  
\end{Corollary}

\begin{Example}[Classical symplectic magnetic terms]
Let $Q$ be any manifold, $\omega\in\Omega^2_{cl}(Q)$ and $D=\Gamma_\omega\subset\mathbb TQ$. Let $L:TQ\to\mathbb R$ be a Lagrangian. In this case $e^\Omega\pi^!D=\Gamma_{\Omega+\pi^*\omega}\subset \mathbb TT^*Q$.

As $\mathcal{H}_{basic}^\bullet(D)=H_{dR}^\bullet (M)$, (cf. Example \ref{sympl-ex}), the 2-form on $D$ admits a basic potential if and only if $\omega$ is de-Rham exact, i.e. $\omega=d\theta$, $\theta\in\Omega^1(Q)$.

Let us assume that the Legendre transform $\mathbb FL:TQ\to T^*Q$ is bijective, and denote the Legendre transform of the Lagrangian by $H$, i.e.

$$H(p)=\langle p, (\mathbb FL)^{-1}p\rangle - L\circ (\mathbb FL)^{-1}(p)$$

In this case $\mathcal DL$ is simply $dH$.
Theorem \ref{theorem-ILS} yields that the critical points of $L(q,\dot q)+\theta(\dot q)$ correspond under the Legendre transform to integral curves of the Hamiltonian flow of $H$ for the symplectic structure $\Omega+\pi^*\omega$. 
Corollary \ref{cor:imp} states that $(X,dH)$  is an implicit Lagrangian system with respect to $e^\Omega\pi^!D$ if and only if the vector field $X$ is the Hamiltonian vector field of $H$ with respect to $\Omega+\pi^*\omega$.

\end{Example}

\begin{Example}[Holonomic constraints as a regular foliation] \label{holonomic}

Let $F\subset TQ$ be a regular foliation. As discussed in Example \ref{reg-fol}, the Dirac structure $D=F\oplus F^\circ$ always admits a {horizontal primitive}, as the 2-form in $\Lambda^2D^*$ is zero (there is no magnetic term). Then $\pi^!D$ is the Dirac structure associated to the pullback foliation $\pi^{-1}(F)$ and 
$$e^\Omega\pi^!D=\{ (w, \alpha)  \in TT^*Q \oplus T^*T^*Q \;|\;   
  \pi_*(w) \in F, \alpha - \Omega^{\flat} w \in \pi^{-1}(F)^\circ    \}$$

Let $L:TQ\to\mathbb R$ be a Lagrangian.
Then Theorem \ref{theorem-ILS} and Corollary \ref{cor:imp} yield that the integral curves of any implicit Lagrangian system $(X,\mathcal DL)$ for $e^\Omega\pi^!D$ are critical points of $L$ among curves that are tangent to $F$. The condition 
\eqref{eq:graphcondition}
translates directly to the Euler-Lagrange equations for a system subject to holonomic constraints, which are classically spelled-out using the Lagrange multipliers \cite{lagrange}. 

\begin{Remark}

 \textbf{Holonomic and non-holonomic constraints.}
Note that the result above concerns the so-called holonomic constraints, i.e. the conditions defining the constraints do not depend essentially on the velocities of the system. Geometrically this means that the foliation $F$ comes from an integrable constraint distribution $\Delta \subset TQ$. Simple mechanical examples and counterexamples can be constructed by ``rolling without slipping'' problems: They are often formulated as an orthogonality condition on the velocity at the contact point -- the condition is integrable for the rolling disk but not for a rolling ball. Under some extra assumptions the non-holonomic constraints can still be treated in the variational approach {(\cite{YoMa}), though with no geometric interpretation. }
In this setting our result is more subtle, since as mentioned above, remarks \ref{almost-dirac} and \ref{almost-holonomic}, the non-integrable almost Dirac structures are very different from the cohomological perspective. {Formally, we cannot speak of an obstruction class, since the ``differential'' does not square to zero. 
However when some primitive can be defined, parts of theorems \ref{thm1} and \ref{theorem-ILS} are still valid.} 
\end{Remark}

\end{Example}

\subsection{Applications to numerics} \label{sec:num}

One of the motivations for the above construction is its potential application to design appropriate structure preserving numerical methods -- so called geometric integrators. 

Historically, the first example of those are the symplectic numerical methods, they are known since several decades, and are now state of the art for Hamiltonian systems (\cite{yoshida}). The key idea is that in the continuous setting the Hamiltonian flow not only preserves the level sets of the Hamiltonian function, but also leaves invariant the symplectic form. It is thus natural to mimic this property for the discrete flow, i.e. computing the trajectory numerically one wants to take the symplectic form into account. And since it is actually the same symplectic form that defines the dynamics of the system, one can reverse the argument: a flow preserving the symplectic form will ``respect'' the level sets of the Hamiltonian defining it. 

The Lagrangian counterpart of this picture is related to so-called variational integrators (\cite{MaWe}), the idea is rather natural as well. Instead of considering a continuous Lagrangian and searching for its extrema along all the paths with fixed endpoints:
\begin{equation} \label{var_pr}
  \inf \mathcal{L} \equiv  \inf \int_{0}^T L(q(t), \dot q(t)) \rd t 
\end{equation}
one defines the discrete version {$L_d$ of the integrant $L$ as follows:}
$$
  L_d (q_{n+1}, q_n, v_n) := \Delta t_n L(q_n, v_n).
$$
Here $q_n \equiv q(t_n)$, $v_n$ is some approximation of $\dot q(t_n)$ depending on $q_n$ and $q_{n+1}$; and $\Delta t_n$ are the time intervals between $q_n$ and $q_{n+1}$, not necessarily all equal. One then defines the discrete analogue of variational principle (DVP), i.e. studies the trajectories $(q_0, q_1 \dots, q_{n-1}, q_n)$  extremizing 
$$
  \mathcal{L}_d = \sum_{n=0}^N L_d, 
$$
subject to $q_0 = q(0)$ and $q_N = q(T)$. 
For conservative mechanical systems one can recover usual symplectic methods with this variational approach, and it is actually more universal, since the timestep is allowed to vary as well. 

A similar strategy can be applied whenever the variational principle can be formulated. For example, in \cite{MaWe} the case of systems with constraints is explored, which motivated some parts of this paper; {later on similar ideas were explored for continuous media problems (see e.g. \cite{gery})}
Hence, the results of Sections \ref{ILS} and \ref{magnetic} on the dynamics on Dirac structures fit to the picture perfectly: they basically say that as soon the cohomological obstruction is absent, one can formulate the Dirac dynamics with a variational approach. In Equation (\ref{var_pr}) one merely replaces the path $q(t)$ in the configuration manifold by a Dirac path $\zeta(t)$. In the continuous setting the Dirac paths preserve the Dirac structure by definition, the variational formulation permits to guarantee this property for the trajectory computed numerically.

There is however an important detail to mention: the folkloric perception of geometric integrators as ``preservation of the geometric structure guarantees preservation of physical properties'' is slightly simplified. For instance in the symplectic case, it is not the original Hamiltonian that is preserved, but its discrete version, for which one can estimate the difference \cite{CHR}. The phenomenon is even more subtle in the variational case. In fact, saying that satisfying the discrete variational principle (DVP) results in preserving some quantities of the system is no longer that straightforward. 
In the generic case the DVP will only give the relations between different variables of the system, but they will still depend on the choice of discretization or approximation of some of them. It may (and often does) also happen that the choice of the discretization to preserve the structure exactly is technically very difficult or even a priori impossible. This means that the correct statements will concern rather preservation of geometric structures up to some order of discretization step.  

{A typical example of this situation is provided by the so-called constraint algorithms: for dynamical systems, the methods to take into account the constraints expressed as algebraic conditions on dynamical variables. When it is impossible to explicitly resolve the constraints, i.e. introduce the dynamical variables satisfying them automatically, there are essentially two approaches: introduce the penalization terms with Lagrange multipliers and discretize them appropriately or ``project'' the solution to the level set of the algebraic conditions at each time step. However, to the best of our knowledge, there are very few proven theorems on how the discrete version of the system satisfies the constraints. We have tried to fill some gaps in empirical observations that one sees in literature. For example (\cite{SH-zamm}) the Dirac structure based algorithm (\cite{YoMa}) in the absence of constraints is naturally symplectic. And some partial results on how to construct \emph{pseudo-geometric} integrators preserving the conditions \emph{up to some order}  are given in \cite{Daria-prog}.}

{With the approach of the current paper we now understand why the naive attempts to increase the order of the constraint-based methods (like e.g. \cite{leok}) do not produce the desired results: roughly speaking the obtained integrators fail to be geometric/variational in the proper sense of the word. A way out would be to formulate the DVP for the calculus of variations in a more general case (\cite{Pontryagin}), and then apply it to the context of Theorems \ref{thm1} and \ref{theorem-ILS}.}

\section*{Perspectives.}

In this paper we have defined the basic Dirac cohomology, which permits {to describe an explicit and verifiable condition for} variational formulation of dynamics on Dirac structures. 

As mentioned in the last part of the paper, on top of purely mathematical interest, this construction should be useful to design more reliable tools for numerical integration of the flow of dynamical systems on Dirac structures. Those in turn naturally appear when studying constraint, interacting or dissipative mechanical systems, which are not in the range of classical Hamiltonian formalism. 
{We expect the results if this paper to provide a unified approach to those and in particular an extension of the observations from Section \ref{sec:num} to arbitrary Dirac structures.}

Let us also mention that, since Poisson manifolds provide an example of Dirac structures, this approach is useful to construct some Poisson integrators. {In the context of this paper there is no conceptual difference between the Dirac structures coming from constraint distributions or from symplectic foliations of Poisson manifolds. The constructed discretizations should thus preserve the symplectic leaves. 
}
This is somewhat complementary to the strategy of \cite{oscar}, where the main tool is rather Hamiltonian dynamics and symplectic groupoids. 

We suppose that these questions {are somewhat technical} and interest a more applied community than the audience of this journal, hence we intend to devote a separate paper {(\cite{discr-pont}) to the description of the discrete variational principal for the general case and the related discussion of the implementation issues.}\\

\textbf{Acknowledgments:} \\
We appreciate inspiring discussions with Dina Razafindralandy, Aziz Hamdouni,  Katarzyna Grabowska, and Pol Vanhaecke at various stages of this work. We are thankful to Tilmann Wurzbacher for valuable comments on the manuscript. \\
We thank the Erwin Schr\"odinger International Institute for Mathematics and Physics for hosting the ``Geometry for Higher Spin Gravity: Conformal Structures, PDEs, and Q-manifolds'' program, that permitted all the authors to gather in the same room and finish the manuscript. \\ 
This work has been supported by the CNRS 80Prime project ``GraNum'' and partially by PHC Procope ``GraNum 2.0''. {L.R. was supported by the RTG2491.}

\appendix

\section{Tulczyjew isomorphism(s)}
\label{tul-iso}
For self-containedness of this paper, we recall here the isomorphisms established by W.~Tulczyjew (\cite{tulcz, tulcz2}) between double (co)tangent bundles (at least one ``co'' should be present). The most non-trivial one is 
$$
\kappa \colon TT^*Q \to T^*TQ,
$$ the construction works as follows: start with the double vector bundle $TTQ$, \\
denote $TQ \overset{p_1}{\to} Q$ and $T(TQ) \overset{p_2}{\to} TQ$ with the respective duality pairings $<\cdot,\cdot>_1$  and $<\cdot,\cdot>_2$. There is a canonical flip $\sigma \colon TTQ \to TTQ$, then the mapping $\kappa$ is (implicitly but canonically) defined by imposing 
$$
  <\kappa(a), b>_2 = <a, \sigma(b)>_1
$$
\begin{Remark}
This flip can be seen as a sort of Schwarz Lemma: for every smooth map $\Sigma(s,t):= \mathbb R^2 \to Q $ (defined in a neighborhood of $(0,0)$):
\begin{enumerate}
    \item $t \mapsto \frac{\partial \Sigma(0,t)}{\partial s} $ is a path in $ TQ$ starting from $\frac{\partial \Sigma}{\partial s}(0,0) $. Its first jet at $0$ belongs to $T_{ \frac{\partial \Sigma}{\partial s}(0,0)}  (TQ) $.
    \item $s \mapsto \frac{\partial \Sigma(s,0)}{\partial t} $ is a path in $ TQ$ starting from $\frac{\partial \Sigma}{\partial t}(0,0) $. Its first jet at $0$ belongs to $T_{ \frac{\partial \Sigma}{\partial t}(0,0)}  (TQ) $.
\end{enumerate}
The canonical flip exchanges both. 
\end{Remark}

All the others are obtained by post- or pre-composing with 
$\Omega^{\flat} \colon TT^*Q \to T^*T^*Q$ or its inverse, where $\Omega$ is the canonical symplectic form on $T^*Q$. 

We will be mostly interested in the isomorphism
$$\beta \equiv \omega^{\flat} \circ  \kappa^{-1} \colon T^*TQ \to T^*T^*Q$$
which is a particular case of the canonical isomorphism also called Tulzcyjew isomorphism $T^*E\simeq T^*E^*$ for any vector bundle $E$ (\cite{MackenzieXu}).

\section{Legendre transformation}

Throughout this section, $E$ is a vector bundle over {$Q$} equipped with a smooth function $ L \colon U \subset E \to \mathbb R $ called \emph{Lagrangian}.

Recall that if a smooth function $f$ on an open convex subset $U \subset V$ of a vector space is strictly convex, then its differential, defined for all $v \in U$ by:
\begin{align*}  \mathbb F f \colon V & \to V^*  
\\
v & \mapsto \left(e \mapsto \left. \frac{d}{dt}\right|_{t=0} f( v + te ) \right) \end{align*} is a diffeomorphism from $ U  $ onto its image. 
If the restriction of $L$ to any fiber is strictly convex, then $\mathbb FL: U \subset E \hookrightarrow E^*$ is a diffeomorphism from $U$ to its image $U'$. \\

We define the Legendre transform $H\in C^\infty(U')$ of $L$ to be the unique function satisfying $H(\alpha)+L(v)=\langle \alpha, v\rangle$ for all $\alpha\in E^*$ with $ \alpha=\mathbb FL(v)$.

\begin{Proposition}
The Legendre transform and the Tulczyjew isomorphism are related by the equality $$\beta(d_eL)=d_{\mathbb FL(v)}H.$$
\end{Proposition}

When $L$ is not strictly convex, the Legendre transform need not exist, however the set $\{\beta(d_eL)~|~e\in U\}$ is a Lagrangian submanifold of $T^*E^*$. It appears throughout the text as the image of $\mathcal DL$.  \\
\newpage

\bibliographystyle{unsrt}
\bibliography{DiracCoh} 

\begin{thebibliography}{10}

\bibitem{RSHD}
Dina Razafindralandy, Vladimir Salnikov, Aziz Hamdouni, and Ahmad Deeb.
\newblock Some robust integrators for large time dynamics.
\newblock {\em Adv. Model. and Simul. in Eng. Sci.}, 6(5), 2019.

\bibitem{morse}
Mar\'{\i}a Barbero Li\~{n}\'{a}n, Hern\'{a}n Cendra, Eduardo Garc\'{\i}a
  Tora\~{n}o, and David Mart\'{\i}n~de Diego.
\newblock Morse families and {D}irac systems.
\newblock {\em J. Geom. Mech.}, 11(4):487--510, 2019.

\bibitem{dirac-algebroids}
Katarzyna Grabowska and Janusz Grabowski.
\newblock Dirac algebroids in {L}agrangian and {H}amiltonian mechanics.
\newblock {\em J. Geom. Phys.}, 61(11):2233--2253, 2011.

\bibitem{GG}
Katarzyna Grabowska and Janusz Grabowski.
\newblock Variational calculus with constraints on general algebroids.
\newblock {\em J. Phys. A}, 41(17):175204, 25, 2008.

\bibitem{tulcz}
W\l odzimierz~M. Tulczyjew.
\newblock Les sous-vari\'{e}t\'{e}s lagrangiennes et la dynamique
  hamiltonienne.
\newblock {\em C. R. Acad. Sci. Paris S\'{e}r. A-B}, 283(1):Ai, A15--A18, 1976.

\bibitem{tulcz2}
W\l odzimierz M.~{Tulczyjew}.
\newblock {Les sous-vari\'et\'es lagrangiennes et la dynamique lagrangienne}.
\newblock {\em {C. R. Acad. Sci., Paris, S\'er. A}}, 283:675--678, 1976.

\bibitem{Mackenzie}
K.~{Mackenzie}.
\newblock {\em {Lie groupoids and Lie algebroids in differential geometry}},
  volume 124.
\newblock Cambridge University Press, Cambridge. London Mathematical Society,
  London, 1987.

\bibitem{vaintrob}
A.~Yu. Va\u{\i}ntrob.
\newblock Lie algebroids and homological vector fields.
\newblock {\em Uspekhi Mat. Nauk}, 52(2(314)):161--162, 1997.

\bibitem{lich}
Andre {Lichnerowicz}.
\newblock {Les vari\'et\'es de Poisson et leurs alg\`ebres de Lie associees}.
\newblock {\em {J. Differ. Geom.}}, 12:253--300, 1977.

\bibitem{Hermann}
R.~{Hermann}.
\newblock {On the accessibility problem in control theory}.
\newblock {Int. Symp. Non-Linear Differ. Equations and Non-Linear Mech.,
  325-332 (1963).}, 1963.

\bibitem{sylvain}
Sylvain Lavau.
\newblock A short guide through integration theorems of generalized
  distributions.
\newblock {\em Differential Geom. Appl.}, 61:42--58, 2018.

\bibitem{ginzburg}
Viktor~L. Ginzburg.
\newblock Equivariant {P}oisson cohomology and a spectral sequence associated
  with a moment map.
\newblock {\em Internat. J. Math.}, 10(8):977--1010, 1999.

\bibitem{zucchini}
Roberto Zucchini.
\newblock The gauging of {BV} algebras.
\newblock {\em J. Geom. Phys.}, 60(11):1860--1880, 2010.

\bibitem{zbMATH00004959}
Theodore~James {Courant}.
\newblock {Dirac manifolds}.
\newblock {\em {Trans. Am. Math. Soc.}}, 319(2):631--661, 1990.

\bibitem{burs}
Henrique Bursztyn.
\newblock A brief introduction to {D}irac manifolds.
\newblock In {\em Geometric and topological methods for quantum field theory},
  pages 4--38. Cambridge Univ. Press, Cambridge, 2013.

\bibitem{dSw}
Ana Cannas~da Silva and Alan Weinstein.
\newblock {\em Geometric models for noncommutative algebras}, volume~10 of {\em
  Berkeley Mathematics Lecture Notes}.
\newblock American Mathematical Society, Providence, RI; Berkeley Center for
  Pure and Applied Mathematics, Berkeley, CA, 1999.

\bibitem{zbMATH06054532}
Camille {Laurent-Gengoux}, Anne {Pichereau}, and Pol {Vanhaecke}.
\newblock {\em {Poisson structures}}, volume 347.
\newblock Berlin: Springer, 2012.

\bibitem{arnold}
V.~I. Arnold.
\newblock {\em Mathematical methods of classical mechanics}, volume~60 of {\em
  Graduate Texts in Mathematics}.
\newblock Springer-Verlag, New York, second edition, 1989.
\newblock Translated from the Russian by K. Vogtmann and A. Weinstein.

\bibitem{blohmann}
Christian Blohmann.
\newblock Removable presymplectic singularities and the local splitting of
  {D}irac structures.
\newblock {\em Int. Math. Res. Not. IMRN}, (23):7344--7374, 2017.

\bibitem{Pontryagin}
L.S. Pontrygin, V.G. Boltyanskii, R.V. Gamkrelidze, and E.F. Mischenko.
\newblock {\em Mathematical Theory of Optimal Processes}.
\newblock M. Nauka, 1983.

\bibitem{lagrange}
Joseph~Louis Lagrange.
\newblock {\em Mécanique Analytique}.
\newblock Mallet-Bachelier, 1855.

\bibitem{YoMa}
Hiroaki Yoshimura and Jerrold~E. Marsden.
\newblock Dirac structures in {L}agrangian mechanics. {I}. {I}mplicit
  {L}agrangian systems.
\newblock {\em J. Geom. Phys.}, 57(1):133--156, 2006.

\bibitem{yoshida}
Haruo Yoshida.
\newblock Construction of higher order symplectic integrators.
\newblock {\em Phys. Lett. A}, 150(5-7):262--268, 1990.

\bibitem{MaWe}
J.~E. Marsden and M.~West.
\newblock Discrete mechanics and variational integrators.
\newblock {\em Acta Numer.}, 10:357--514, 2001.

\bibitem{gery}
Xiaodan Cao, Abdelbacet Oueslati, An~Danh Nguyen, and G\'{e}ry de~Saxc\'{e}.
\newblock Numerical simulation of elastoplastic problems by
  {B}rezis-{E}keland-{N}ayroles non-incremental variational principle.
\newblock {\em Comput. Mech.}, 65(4):1005--1018, 2020.

\bibitem{CHR}
Dina Razafindralandy, Aziz Hamdouni, and Marx Chhay.
\newblock {A review of some geometric integrators}.
\newblock {\em {Advanced Modeling and Simulation in Engineering Sciences}},
  5(1):16, December 2018.

\bibitem{SH-zamm}
Vladimir Salnikov and Aziz Hamdouni.
\newblock From modelling of systems with constraints to generalized geometry
  and back to numerics.
\newblock {\em ZAMM Z. Angew. Math. Mech.}, 99(6):e201800218, 13, 2019.

\bibitem{Daria-prog}
Daria Loziienko, Aziz Hamdouni, and Vladimir Salnikov.
\newblock {Construction of pseudo-geometric integrators}.
\newblock {\em {Program Comput Soft}}, 2, 2022.

\bibitem{leok}
Melvin Leok and Tomoki Ohsawa.
\newblock Discrete {D}irac structures and implicit discrete {L}agrangian and
  {H}amiltonian systems.
\newblock In {\em X{VIII} {I}nternational {F}all {W}orkshop on {G}eometry and
  {P}hysics}, volume 1260 of {\em AIP Conf. Proc.}, pages 91--102. Amer. Inst.
  Phys., Melville, NY, 2010.

\bibitem{oscar}
Oscar Cosserat.
\newblock {Symplectic groupoids for Poisson integrators}.
\newblock {\em {Preprint arXiv:2205.04838}}, 2022.

\bibitem{discr-pont}
Aziz Hamdouni, Alexei Kotov, Camille Laurent-Gengoux, and Vladimir Salnikov.
\newblock Discrete pontryagin's maximum principle and applications.
\newblock {\em in prepration}, 2022.

\bibitem{MackenzieXu}
Kirill C.~H. Mackenzie and Ping Xu.
\newblock Lie bialgebroids and {Poisson} groupoids.
\newblock {\em Duke Math. J.}, 73(2):415--452, 1994.

\end{thebibliography}


\end{document}